\documentclass[a4paper,10pt]{amsart}
\usepackage{amsmath,amsfonts,amssymb, amsthm, xypic}
\usepackage{epsfig,psfrag}
\usepackage[all]{xy}

\def\A{\boldsymbol{\mathcal{E}}}
\def\Ak{\boldsymbol{\mathcal{E}}_{\mathbf{K}}}

\def\tA{\widetilde{\boldsymbol{\mathcal{E}}}}

\def\Z{\mathbb{Z}}
\def\TT{\mathbb{T}}

\def\K{\mathbf{K}}
\def\H{\mathbf{H}}
\def\U{\mathbf{U}}
\def\SH{\mathbf{S}\ddot{\mathbf{H}}}
\def\qed{$\hfill \checkmark$}
\def\N{\mathbb{N}}
\def\tto{\twoheadrightarrow}

\def\a{\alpha}

\def\C{\mathbb{C}}

\def\x{\mathbf{x}}
\def\y{\mathbf{y}}
\def\w{\mathbf{w}}
\def\z{\mathbf{z}}
\def\p{\mathbf{p}}

\def\v{\nu}
\def\t{\tau}
\def\ut{\underline{u}}
\def\tu{\underline{u}}

\newtheorem{theo}{\bf{Theorem}}
\newtheorem{lem}{Lemma}
\newtheorem{cor}{Corollary}
\newtheorem{prop}{Proposition}

\numberwithin{lem}{section}
\numberwithin{conj}{section}
\numberwithin{prop}{section}
\numberwithin{cor}{section}
\numberwithin{equation}{section}
\numberwithin{claim}{section}

\title{Drinfeld realization of the elliptic Hall algebra}
\author{Olivier Schiffmann}
\thanks{\noindent *
Institut Math\'ematique de Jussieu, 175 rue du Chevaleret, 75013 Paris, FRANCE, 
email:\;\texttt{olive@math.jussieu.fr}}

\begin{document}
\maketitle

\begin{abstract}
We give a new presentation of the Drinfeld double $\A$ of the (spherical) elliptic Hall algebra $\A^+$ introduced in our previous work \cite{BS}.
This presentation is similar in spirit to Drinfeld's `new realization' of quantum affine algebras. This answers, in the case of elliptic curves, a question
of Kapranov concerning functional relations satisfied by (principal, unramified) Eisenstein series for $GL(n)$ over a function field. It also provides
proofs of some recent conjectures of Feigin, Feigin, Jimbo, Miwa and Mukhin (\cite{F2JM2}).
\end{abstract}

\tableofcontents
\setlength{\unitlength}{10pt}

\section{Introduction}

\vspace{.15in}

In the seminal paper \cite{Kap}, M. Kapranov initiated the study of the so-called Hall algebra $\H_X$ of the category of coherent sheaves over a smooth
projective curve $X$ (defined over a finite field). It may be interpreted in the language of automorphic forms as an algebra of (unramified) 
automorphic forms for the collection of groups $GL(n)$ ($n \geq 1$) over the function field of $X$, with the product being given by the parabolic induction
$GL(n) \times GL(m) \to GL(n+m)$. On the other hand, the well-known analogy between Hall algebras of abelian categories and quantum groups
(discovered by Ringel \cite{Ringel}) suggests that $\H_X$ should behave like the positive half of some `quantum envelopping algebra'. Indeed, as observed by Kapranov, the functional equations for Eisenstein series (which express the invariance of these Eisenstein series under the action of the Weyl group, and which are quadratic relations) resemble very much the relations appearing in Drinfeld's 'new realization' of quantum affine algebras (see \cite{Drin}). In fact, they coincide precisely with these relations when $X=\mathbb{P}^1$ (see \cite{Kap}, \cite{BK}). As soon as $X$ is of genus at least one, however, the functional equations do \textit{not} suffice to determine the Hall algebra $\H_X$, and one has to look for some new, higher rank `functional equations' satisfied by Eisenstein series.

\vspace{.1in}

The aim of this paper is to determine these higher rank functional equations when $X$ is an elliptic curve. More precisely, we consider the Drinfeld double $\U_X$ of the \textit{spherical} Hall algebra $\U_X^+$ of $X$, i.e the (Drinfeld double of the) subalgebra of $\H_X$ generated by the characteristic functions of the connected components of the Picard groups $Pic^d(X)$ together with the characteristic functions of the stacks of torsion sheaves on $X$. We show that $\U_X$ can be
abstractly presented as an algebra generated by the above elements subject to the standard (quadratic) functional equations plus one set of \textit{cubic} relations (see Theorem~\ref{T:3}, equations (\ref{E:NEW})). Our method is rather brutal and very specific to elliptic curves; it is based on a combinatorial realization of $\U_X$ given in our previous work \cite{BS} in terms of lattice paths in $\Z^2$. We nevertheless expect that our result can be extended to the whole Hall algebra $\H_X$ (of an elliptic curve) and its Drinfeld double $\mathbf{D}\H_X$.

\vspace{.1in}

As it turns out, the (spherical) elliptic Hall algebra $\U_X$, or rather its rational form $\A$ appears in several other guises : as shown in \cite{SV} it projects to the spherical Cherednik algebra
$\SH_n$ for any $n \in \Z$ and as such may be thought of as the stable limit $\SH_{\infty}$ of these Cherednik algebras; as shown in \cite{SV2}, $\A$ may be identified with a convolution algebra acting on the (equivariant) K-theory of Hilbert schemes of points in $\C^2$ and its positive half $\A^>$ admits a realization as
a shuffle algebra of Feigin-Odesskii type. Finally, the algebra $\A$ (and the presentation given in Theorem~\ref{T:3}) also appears in recent work of Feigin, Feigin, Jimbo, Miwa and Mukhin (see \cite{F2JM2}, \cite{F2JM2II} and \cite{FT}) where its representation theory is studied. Our results verify some of the conjectures
presented in \textit{loc. cit} (see Section~4.2).

\vspace{.1in}

The plan of the paper is the following~: after introducing the algebra $\U_X$, its rational form $\A$ and after recalling their relation to DAHAs and shuffle algebras (Sections 1-3) we state our main theorem giving a Drinfeld realization of $\A$ in Section~4. Section~5 is dedicated to the (elementary but intricate) proof of that result. We conclude with some formulas for the Hopf algebra structure of $\A$, in terms of our new `Drinfeld' realization. 

\vspace{.2in} 

\section{Recollections on the elliptic Hall algebra}

\vspace{.1in}

\paragraph{\textbf{1.1.}} We begin with some recollections concerning the spherical Hall algebra $\U_X^+$ of an elliptic curve $X$ defined over a finite field $\mathbb{F}_q$ and its Drinfeld double $\U_X$. We refer to \cite{BS} for details. Let $Coh(X)$ be the category of coherent sheaves over $X$, and let $\mathcal{I}$ be the set of isomorphism classes of objects in $Coh(X)$. There is a partition $\mathcal{I}=\sqcup_{r,d} \mathcal{I}_{r,d}$ according to the rank and the degree of a coherent sheaf.
Set 
$$\H_X[r,d]=\{f : \mathcal{I}_{r,d} \to \C\;|\; \#(supp\;f)<\infty\} =\bigoplus_{\mathcal{F} \in \mathcal{I}_{r,d}} \C 1_{\mathcal{F}},$$
and put $\H_{X}=\bigoplus_{r,d} \H_{\x}[r,d]$, where the sums runs over all possible pairs $(r,d)$, i.e. over $(\Z^2)^+=\{(r,d)\:|\; r >0\;\text{or}\;r=0, d >0\}$.
To a triple $(\mathcal{F},\mathcal{G},\mathcal{H})$ of coherent sheaves we associate the finite set
$$\mathcal{P}_{\mathcal{F},\mathcal{G}}^\mathcal{H}=\{(\phi, \psi)\;|\;
0 \to \mathcal{G} \stackrel{\phi}{\to} \mathcal{H} \stackrel{\psi}{\to} \mathcal{F} \to 0\; \text{is}\;\text{exact}\}$$
and we set ${P}_{\mathcal{F},\mathcal{G}}^\mathcal{H}= \# \mathcal{P}_{\mathcal{F},\mathcal{G}}^\mathcal{H}$.
We write $a_{\mathcal{K}}=
\#\mathrm{Aut}(\mathcal{K})$ for a coherent sheaf  $\mathcal{K}$. Define an associative product
on $\mathbf{H}_{X}$ by the formula
\begin{equation}\label{E:Hallprod}
1_{\mathcal{F}}\cdot 1_{\mathcal{G}}=v^{-\langle \mathcal{F},\mathcal{G} \rangle}
\sum_{\mathcal{H}} \frac{1}{a_{\mathcal{F}}a_{\mathcal{G}}}{P}_{\mathcal{F},\mathcal{G}}^\mathcal{H} 1_{\mathcal{H}},
\end{equation}
and a coassociative coproduct by the formula
\begin{equation}\label{E:Hallcoprod}
\Delta\bigl(1_{\mathcal{H}}\bigr)=\sum_{\mathcal{F},\mathcal{G}} v^{-\langle \mathcal{F},\mathcal{G}\rangle}\frac{{P}_{\mathcal{F},\mathcal{G}}^\mathcal{H}}{a_{\mathcal{H}}}
 1_{\mathcal{F}} \otimes 1_{\mathcal{G}}.
\end{equation}
There is a natural pairing on $\H_X$ given by
$$\bigl(1_{\mathcal{F}},1_{\mathcal{G}}\bigr)=\delta_{\mathcal{F},\mathcal{G}}
\frac{1}{a_{\mathcal{F}}}.$$
Then $(\H_X, \cdot, \Delta)$ is a (topological) bialgebra (see e.g. \cite{SLectures}). Moreover, the pairing $(\;,\;)$ is a Hopf pairing and it is nondegenerate, making $\H_X$ a self-dual bialgebra. 

\vspace{.1in}

For $d \in \Z$ let $1_{Pic^d(X)}=\sum_{\mathcal{L} \in Pic^d(X)} 1_{\mathcal{L}}$ be the characteristic function of the Picard variety of degree $d$; for $l \geq 1$ let
$1_{Tor^l(X)}=\sum_{\mathcal{T} \in Tor^l(X)}1_{\mathcal{T}}$ be the characteristic function of the set of torsion sheaves of degree $l$; more generally, for any $(r,d) \in (\Z^2)^+$ let $1^{ss}_{r,d}$ be the characteristic function of the set of semistable sheaves of rank $r$ and degree $d$.  By definition, the \textit{spherical Hall algebra} $\U^+_X$ of $X$ is the subalgebra of $\H_X$ generated by $\{1_{Pic^d(X)}\;|d \in \Z\} \cup \{1_{Tor^l(X)}\;|\; l \in \N\}$. One shows that it is a sub bialgebra of $\H_X$, and that it contains the element $1^{ss}_{r,d}$ for any $r,d$. 

\vspace{.1in}

Finally, let $\U_X$ be the Drinfeld double of $\U^+_X$; it is a an algebra generated by two copies $\U_X^+$ and $\U_X^-$ of the spherical Hall algebra $\U^+_X$ (see e.g. \cite[Section~3]{BS} for more details). If $u$ belongs to the spherical Hall algebra, we denote by $u^+$ and $u^-$ the corresponding elements in $\U^+_X$ and $\U^-_X$.

\vspace{.1in}

We refer to \cite{Kap} for the interpretation of $\H_X$ or $\U_X$ in the language of automorphic forms over function fields.

\vspace{.2in}

\paragraph{\textbf{1.2.}} The spherical Hall algebras $\U^+_X$ and $\U_X$ admit the following combinatorial presentation, obtained in \cite{BS}. Let $\sigma, \overline{\sigma}$ be the two Weil numbers of $X$, i.e. the Frobenius eigenvalues in $H^1(\overline{X}, \overline{\mathbb{F}_q})$.
For $i \geq 1$ we put
$$\a_i=\a_i(\sigma, \overline{\sigma})=(1-\sigma^i)(1-\overline{\sigma}^i)(1-(\sigma \overline{\sigma})^{-i})/i.$$

\vspace{.1in}

Let us set $(\Z^2)^*=\Z^2 \backslash \{(0,0)\}$,
$$(\Z^2)^+=\{ (p,q) \in \Z^2\;|\; p > 0 \;\text{or}\; p=0, q>0\},$$
$$(\Z^2)^>=\{ (p,q) \in \Z^2\;|\; p > 0\}.$$
and $ (\Z^2)^-=-(\Z^2)^+, (\Z^2)^<=-(\Z^2)^>$. 
For any $\x=(p,q) \in (\Z^2)^*$ we put $deg(x)=g.c.d. (p,q) \in \N$.
Finally, for a pair of non-collinear 
$(\x,\y) \in (\Z^2)^*$ we set $\epsilon_{\x,\y}=sign(det(\x,\y)) \in \{\pm 1\}$. We also let ${\Delta}_{\x,\y}$ stand for the triangle in $\Z^2$ with vertices $\{(0,0), \x, \x+\y\}$.

\centerline{
\begin{picture}(28,12)
\multiput(6,3.5)(2,0){8}{\circle*{.2}}
\multiput(6,1.5)(2,0){8}{\circle*{.2}}
\multiput(6,5.5)(2,0){8}{\circle*{.2}}
\multiput(6,7.5)(2,0){8}{\circle*{.2}}
\multiput(6,9.5)(2,0){8}{\circle*{.2}}
\put(6,5.5){\circle*{.3}}
\put(8,9.5){\circle*{.3}}
\put(14,3.5){\circle*{.3}}
\put(16,7.5){\circle*{.3}}
\put(5,6){\tiny{$(0,0)$}}
\put(7,8.5){{$\mathbf{y}$}}
\put(13,2.5){{$\mathbf{x}$}}
\put(16.5,6.5){{$\mathbf{x}+\y$}}
\put(6,5.5){\line(5,1){10}}
\put(6,5.5){\line(4,-1){8}}
\put(14,3.5){\line(1,2){2}}
\end{picture}}
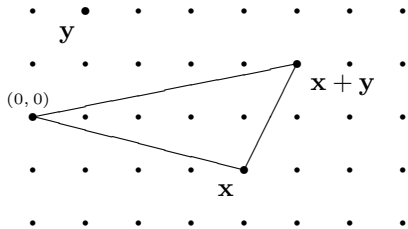
\centerline{Figure 1. The triangle ${\Delta}_{\x,\y}$}

\vspace{.15in}

\noindent
\textbf{Definition.} Let $\A_X$ be the $\C$-algebra generated by elements $\{u_{\x}\;|\x \in (\Z^2)^*\}$ modulo the following set of relations~:
\begin{enumerate}
\item[i)] If $\x,\x'$ belong to the same line in $\Z^2$ then
$$[u_\x,u_{\x'}]=0,$$
\item[ii)] Assume that $\x,\y \in (\Z^2)^*$ are such that $deg(\x)=1$ and that 
${\Delta}_{\x,\y}$ has no interior lattice point. Then 
$$[u_\y,u_{\x}]=\epsilon_{\x,\y}
\frac{\theta_{\x+\y}}{\a_1}$$
where the elements $\theta_{\z}$, $\z \in (\Z^2)^*$ are obtained by equating the Fourier
coefficients of the collection of relations
\begin{equation}\label{E:formulatheta}
\sum_i \theta_{i\x_0}s^i=exp\bigg(\sum_{r \geq 1}\a_ru_{r\x_0}s^r\bigg),
\end{equation}
for any $\x_0 \in (\Z^2)^*$ such that $deg(\x_0)=1$ (observe that $\theta_{\z}/\a_1=u_{\z}$ if $deg(\z)=1$). 
\end{enumerate}

\vspace{.15in} 

The relation with spherical Hall algebras is given in the following theorem. Define elements $T_{r,d}$ in $\U^+_X$ for $(r,d) \in (\Z^2)^+$ via the following equations
$$1+\sum_{l \geq 1} 1^{ss}_{l\x}=exp\bigg(\sum_{l \geq 1} \frac{1}{[l]}T_{l\x} s^l\bigg)$$
for $\x$ a \textit{primitive} vector. We also set $c_i=\a_i/(v^{-1}-v)$.

\vspace{.1in}

\begin{theo}[\cite{BS}] The assignement 
$$T_{1,d}^\pm \mapsto c_1u_{\pm 1, \pm d}, \qquad T_{0,l}^{\pm} \mapsto c_l u_{0, \pm l}$$
extends to an isomorphism of algebras $\Phi: \U_X \stackrel{\sim}{\to} \A_X$. Moreover we have $\Phi(T_{\x}^\pm)= c_{deg(\x)} u_{\x}$ for any $\x \in (\Z^2)^+$.
\end{theo}

\vspace{.15in}

Observe that $\A_X$ carries a natural $SL(2,\Z)$-symmetry. It corresponds to an action of $SL(2,\Z)$ by Fourier-Mukai transforms on the derived category $D^b(Coh(X))$. Rather than $\A_X$, we will use the rational form $\Ak$ of $\A_X$ defined over the field $\mathbf{K}=\C(\sigma, \overline{\sigma})$, where $\sigma, \overline{\sigma}$ are now formal parameters. Since the ground field will always be $\mathbf{K}$, we will drop the index $\mathbf{K}$ from the notation. 

\vspace{.1in}

We denote by $\A^{\pm}$ the subalgebra of $\A$ generated by $u_{\x}$ for $\x \in (\Z^2)^{\pm}$. It is shown in \cite[Section~5]{BS}. that relations i) and ii), restricted to $(\Z^2)^{\pm}$ give a presentation of $\A^{\pm}$, and that there is a decomposition
\begin{equation}\label{E:D1}
\A \simeq \A^+  \otimes \A^-.
\end{equation}
It is also shown in \cite[Section~5]{BS} that $\A$ is actually generated by elements $u_{r,d}$ with $-1 \leq r \leq 1$; similar results hold for the subalgebras $\A^+$ and $\A^-$.

\vspace{.1in}

We will need to refine (\ref{E:D1}) slightly~: let $\A^{>}$, resp. $\A^{<}$, resp. $\A^{0}$, be the subalgebras of $\A$ generated by the elements $\{{u}_{1,l}\;|\; l \in \Z\}$, resp. $\{{u}_{-1,l}\;|\; l \in \Z\}$, resp. $\{{u}_{0,l}\;|\; l \in \Z^*\}$. Then from (\ref{E:D1}) and the defining relations of $\A$ one deduces the triangular decomposition \begin{equation}\label{E:D2}
\A \simeq \A^{>}  \otimes \A^{0} \otimes \A^{<}
\end{equation}
(see the proof of Lemma~\ref{L:1} below).

\vspace{.2in} 

\section{Link to spherical DAHAs}

\vspace{.15in}

\paragraph{\textbf{2.1.}}In this section, which is included for completeness, we briefly describe the relation between the elliptic Hall algebra $\A$ and
spherical DAHAs of type $GL(n)$. Let $\ddot{\H}_n$ be the double affine Hecke algebra of type $GL(n)$ and parameters
$t=\sigma^{-1}, q=\overline{\sigma}^{-1}$. This is an algebra generated by elements $X_i^{\pm 1}, Y_i^{\pm 1}, T_j$ for $i=1, \ldots, n$ and $j=1, \ldots, n-1$,
subject to a set of relations which we won't write here (see \cite{SV}, Section~2). Let $S$ stand for the maximal idempotent in the finite Hecke
algebra $\H_n$ generated by $T_1, \ldots, T_{n-1}$ and put $\SH=S \ddot{\H}_n S$. There is a well-known action of the group $SL(2,\Z)$ on $\SH_n$ by algebra automorphisms, and we may use it to define a unique collection of elements $P^{(n)}_{r,d}$ for $(r,d) \in \Z^2 \backslash \{0\}$ satisfying
$P^{(n)}_{0,d}=S \sum_i Y_i^d S$ if $d >0$ and
$$P^{(n)}_{\gamma(r,d)}=\gamma \cdot P^{(n)}_{r,d}$$
for all $\gamma \in SL(2,\Z)$.

\vspace{.1in}

The following result is proved in \cite[Theorem~3.1]{SV}~:

\vspace{.1in}

\begin{theo}\label{T:1} For any $n$ there is a surjective algebra morphism 
$\Psi_n~: \A \tto \SH_n$ given by $$ u_{r,d} \mapsto \frac{1}{q^{deg(r,d)}-1}P^{(n)}_{r,d}.$$ 
Moreover we have
$\bigcap_{n} Ker\; \Psi_n=\{0\}$.
\end{theo}

We may think of $\A$ as a stable limit $\SH_{\infty}$ of $\SH_n$ as $n$ tends to infinity.

\vspace{.2in} 

\section{Shuffle algebra presentation}

\vspace{.15in}

\paragraph{\textbf{3.1.}} We now provide a shuffle realization of the \textit{positive} elliptic Hall algebra $\A^>$. Such a realization was obtained in \cite{SV2}. The same 
shuffle algebra also appears in the work of Feigin and Tsymbaliuk (see \cite{FT}).

\vspace{.1in}

Set 
$$\tilde{\zeta}(z)=\frac{(1-\sigma z)(1-\overline{\sigma}z)(1-(\sigma \overline{\sigma})^{-1}z)}{(1-z)}$$
and put $\tilde{\zeta}(z_1, \ldots, z_r)=\prod_{i<j} \tilde{\zeta}(z_i/z_j)$. Following Feigin and Odesskii (\cite{FO}) we define an associative algebra $\mathbf{A}=\mathbf{A}_{\tilde{\zeta}(z)}$ as follows. Consider the twisted symmetrization operator 
\begin{equation*}
\begin{split}
\Psi_r~: \K[z_1^{\pm 1}, \ldots, z_r^{\pm 1}] & \to \K[z_1^{\pm 1}, \ldots, z_r^{\pm 1}]^{\mathfrak{S}_r}\\
P(z_1, \ldots, z_r) &\mapsto \sum_{\gamma \in \mathfrak{S}_r} \gamma \cdot \big( \tilde{\zeta}(z_1, \ldots, z_r) P(z_1, \ldots, z_r)\big)
\end{split}
\end{equation*} 
and set $\mathbf{A}_r=Im(\Psi_r)$. Note that $\Psi_r$ is a $\K[z_1^{\pm 1}, \ldots, z_r^{\pm 1}]^{\mathfrak{S}_r}$-module morphism. There is a unique
linear map $m_{r,s}: \mathbf{A}_r \otimes \mathbf{A}_s \to \mathbf{A}_{r+s}$ fitting in a commutative diagram
\begin{equation}\label{E:shufflediagram}
\xymatrix{ \K[z_1^{\pm 1}, \ldots, z_r^{\pm 1}] \otimes \K[z_1^{\pm 1}, \ldots, z_{s}^{\pm 1}] \ar[r]^-{\Psi_r \otimes \Psi_{s}} \ar[d]_-{i_{r,s}}^{\sim} & \mathbf{A}_r \otimes \mathbf{A}_{s} \ar[d]_-{m_{r,s}} \\ \K[z_1^{\pm 1}, \ldots, z_{r+s}^{\pm 1}] \ar[r]^-{\Psi_{r+s}} & \mathbf{A}_{r+s}}
\end{equation}
where $i_{r,s}\big( P(z_1, \ldots, z_r) \otimes Q(z_1, \ldots, z_{s})\big)=
P(z_1, \ldots, z_r)Q(z_{r+1}, \ldots, z_{r+s})$.
It is easy to check that the maps $m_{r,s}$ endow the space
$\mathbf{A}=\K  \oplus \bigoplus_{r \geqslant 1} \mathbf{A}_r$
with the structure of an associative algebra.
The product in $\mathbf{A}$ may be explicitly written
as the shuffle operation
\begin{equation*}
\begin{split}
h(z_1, \ldots, z_r) \cdot &f(z_1, \ldots, z_{a})\\
&=\frac{1}{r!s!} Sym_{r+s} \bigg( \hspace{-.1in}\prod_{\substack{1 \leqslant i \leqslant r\\ r+1 \leqslant j \leqslant r+s}} \hspace{-.15in}\tilde{\zeta}(z_i/z_j)\cdot \;h(z_1, \ldots, z_r) f(z_{r+1}, \ldots, z_{r+s})\bigg)
\end{split}
\end{equation*}
where $Sym$ is the usual symmetrization operator. The following is shown in \cite[Theorem~10.1]{SV2}~:

\vspace{.1in}

\begin{theo}\label{T:2}
The assignement $u_{1,d} \mapsto z_1^d \in \mathbf{A}_1$ for $d \in \Z$ extends to an isomorphism of $\K$-algebras $\Phi: \A^> \stackrel{\sim}{\to} \mathbf{A}$.
\end{theo}

\vspace{.1in}

\noindent
{\textbf{Remark.}} A similar shuffle realization exists for the (spherical) Hall algebra of an arbitrary smooth projective curve (see \cite{SV3}).

\vspace{.2in} 

\section{Drinfeld presentation}

\vspace{.15in}

This section contains our main result, i.e. a presentation of the elliptic Hall algebra $\A$ akin to the 'Drinfeld new realization' of quantum affine algebras.

\vspace{.1in}

\paragraph{\textbf{4.1.}} Consider the formal series
$$\mathbb{T}_1(z)=\sum_l u_{1,l}z^l, \qquad \mathbb{T}_{-1}(z)=\sum_l u_{-1,l}z^l,$$
and
$$\mathbb{T}^+_0(z)=1+\sum_{l \geq 1} \theta_{(0,l)}z^l, \qquad \mathbb{T}^-_0(z)=1+\sum_{l \geq 1} \theta_{(0,-l)}z^{-l}. $$
We introduce the polynomials
$$\chi_1(z,w)=(z-\sigma w)(z-\overline{\sigma}w)(z-(\sigma\overline{\sigma})^{-1}w),$$
$$\chi_{-1}(z,w)=(z-\sigma^{-1}w)(z-\overline{\sigma}^{-1}w)(z-\sigma\overline{\sigma}w).$$
Note that $\chi_{-1}(z,w)=-\chi_1(w,z)$. Finally, we set as usual $\delta(x)=\sum_{l \in \Z} x^l$.

\vspace{.1in}

In the theorem below, the equations are to be understood formally, i.e. as equalities of Fourier coefficients of $z^nw^l$ for all $n, l \in \Z$. The residue operation
is defined as follows~: if $A(z)=\sum_{l \in \Z} a_l z^l$ is a formal series then $Res_z A(z)=a_{-1}$.

\vspace{.15in}

\begin{theo}\label{T:3} The elliptic Hall algebra $\A$ is isomorphic to the algebra generated by the Fourier coefficients of $\TT_{-1}(z), \TT_1(z), \TT_0^+(z)$ and $\TT_0^-(z)$, modulo the following relations, for all $\epsilon, \epsilon_1, \epsilon_2 \in \{1,-1\}$~:
\begin{equation}\label{E:1}
\mathbb{T}_0^{\epsilon_1}(z)\mathbb{T}_0^{\epsilon_2}(w)=\mathbb{T}_0^{\epsilon_2}(w)
\mathbb{T}_0^{\epsilon_1}(z),
\end{equation}
\begin{equation}\label{E:2}
\chi_{\epsilon_1}(z,w)\mathbb{T}^{\epsilon_2}_0(z) \mathbb{T}_{\epsilon_1}(w)=\chi_{-\epsilon_1}(z,w)\mathbb{T}_{\epsilon_1}(w) \mathbb{T}^{\epsilon_2}_0(z),
\end{equation}
\begin{equation}\label{E:3}
\chi_{\epsilon}(z,w)\mathbb{T}_{\epsilon}(z) \mathbb{T}_{\epsilon}(w)=\chi_{-\epsilon}(z,w)\mathbb{T}_{\epsilon}(w) \mathbb{T}_{\epsilon}(z),
\end{equation}
\begin{equation}\label{E:5}
[\mathbb{T}_{-1}(z), \mathbb{T}_1(w)]=\frac{1}{\a_1}\left(\mathbb{T}^-_0(z)\delta\left(\frac{z}{w}\right) -\mathbb{T}^+_0(z)\delta\left(\frac{z}{w}\right)\right),
\end{equation}
together with the cubic relations
\begin{equation}\label{E:NEW}
Res_{z,y,w} \big[(zyw)^m  (z+w)(y^2-zw)\mathbb{T}_\epsilon(z)\mathbb{T}_\epsilon(y)\mathbb{T}_\epsilon(w)\big]=0,\\
\end{equation}
for all $m \in \Z$ and $\epsilon \in \{-1,1\}$.
\end{theo}

\noindent
The proof of the above theorem is given in Section~3. The cubic relations (\ref{E:NEW}) may also be written more simply as follows :
$$[[u_{1,l+1},u_{1,l-1}],u_{1,l}]=[[u_{-1,l-1},u_{-1,l+1}],u_{-1,l}]=0$$
for all $l \in \Z$.
Observe that the $SL(2,\Z)$-symmetry of $\A$ is broken in the above presentation. However, there is still a natural symmetry by the unipotent subgroup $\tiny{{\begin{pmatrix} 1 & \Z \\ 0 & 1 \end{pmatrix}}} \subset SL(2,\Z)$, given by 
\begin{equation}\label{E:unipsym}
{\tiny{{\begin{pmatrix} 1 & n \\ 0 & 1 \end{pmatrix}}}}\cdot u_{\pm 1,l}=u_{\pm 1,l \pm n}, \qquad {\tiny{{\begin{pmatrix} 1 & n \\ 0 & 1 \end{pmatrix}}}} \cdot \theta_{0,l}=\theta_{0,l}.
\end{equation}

\vspace{.2in}

\paragraph{\textbf{4.2.}} The algebra generated by elements $u_{\pm 1,d}$, $\theta_{0,d}$ for $d \in \Z$ subject to the relations (\ref{E:1} --\ref{E:NEW}) was considered in \cite{F2JM2}, where it was coined `quantum continuous $\mathfrak{gl}_{\infty}$'. There the authors conjectured its link to the stable limit $\SH_{\infty}$ of spherical DAHAs on the one hand, and to the shuffle algebra $\mathbf{A}_{\tilde{\zeta}(z)}$ on the other. The combination of Theorems~\ref{T:1},\ref{T:2} with Theorem~\ref{T:3} yields a proof of these conjectures.

\vspace{.2in}

\section{Proof of the main theorem}

\vspace{.1in}

This section contains the details of the proof of Theorem~\ref{T:3}.

\vspace{.15in}

\paragraph{\textbf{5.1.}} We have to introduce several combinatorial notions~:

\vspace{.1in}

\noindent
\textit{Segments and slopes.} We will sometimes call \textit{segment} a nonzero element of $\Z^2$. For a segment $\z$, let $\mu(\z) \in ]-\pi/2, 3\pi/2]$ stand for the angle between the line through $\z$ and the x-axis in $\Z^2$. We will call $\mu(\z)$ the \textit{slope} of $\z$.

\vspace{.1in}

\noindent
\textit{Paths and sequences.} A \textit{sequence} in $\Z^2$ is a finite ordered set $(\x_1, \ldots, \x_n)$ of elements of $\Z^2$. It is helpful to visualize a sequence $(\x_1, \ldots, \x_n)$ as the broken line in $(\Z^2)$ starting at the origin and connecting the vertices $\x_1$, $\x_1+\x_2, \ldots, \x_1+\cdots +\x_n$. We say that a sequence $\mathbf{s}=(\x_1, \ldots, \x_n)$ is equivalent to $\mathbf{s}'=(\x_1', \ldots, \x'_n)$ if $\mathbf{s}'$ can be obtained from $\mathbf{s}$ by successively permuting adjacent entries $\x_i, \x_{i+1}$ \textit{of the same slope}.
We define a \textit{path} in $\Z^2$ to be an equivalence class of sequences as above. We will denote by $\textbf{Path}$ the set of all paths. If $U \subset (\Z^2)^*$ is any subset we will say that a path $\p$ is in $U$ if all of its entries lie in $U$, and we denote by $\textbf{Path}^U$ the set of all such paths.
 
\vspace{.1in}

\noindent
\textit{Convex paths.}
We will say that a path $\p=(\x_1, \ldots, \x_n)$ is \textit{convex} if it satisfies
 $$-\pi/2 < \mu(\x_1) \leq \mu(\x_2) \leq \cdots \leq \mu(\x_n) \leq 3\pi/2,$$
 and we denote by $\textbf{Conv}$ the set of all such paths.  A \textit{(local) convexification} of a path $\p=(\x_1, \ldots, \x_n)$ is a path obtained from $\p$ by replacing two adjacent segments $\x_i, \x_{i+1}$ forming a non-convex subpath (i.e. satisfying $\mu(\x_i) > \mu(\x_{i+1})$) by a convex path lying in the triangle $\Delta$ with sides $\x_i, \x_{i+1}, \x_i+\x_{i+1}$ (see figure 2.). 
 
\centerline{
\begin{picture}(25,10)
\multiput(0,0.5)(1,0){11}{\circle*{.2}}
\multiput(0,1.5)(1,0){11}{\circle*{.2}}
\multiput(0,2.5)(1,0){11}{\circle*{.2}}
\multiput(0,3.5)(1,0){11}{\circle*{.2}}
\multiput(0,4.5)(1,0){11}{\circle*{.2}}
\multiput(0,5.5)(1,0){11}{\circle*{.2}}
\multiput(0,6.5)(1,0){11}{\circle*{.2}}
\multiput(0,7.5)(1,0){11}{\circle*{.2}}
\multiput(0,8.5)(1,0){11}{\circle*{.2}}
\put(1,5.5){\circle*{.3}}
\put(0,5.85){\tiny{$(0,0)$}}
\thicklines
\put(1,5.5){\line(1,0){1}}
\put(2,5.5){\circle*{.3}}
\put(2,5.5){\line(1,1){2}}
\put(2,5.5){\line(2,-1){4}}
\put(4,7.5){\circle*{.3}}
\put(4,7.5){\line(1,-2){2}}
\put(2,5.5){\line(1,-2){2}}
\put(4,1.5){\line(1,1){2}}
\put(6,3.5){\circle*{.3}}
\put(6,3.5){\line(2,1){2}}
\put(8,4.5){\circle*{.3}}
\put(8,4.5){\line(2,-1){2}}
\put(10,3.5){\circle*{.3}}
\put(10,3.5){\line(0,1){2}}
\put(10,5.5){\circle*{.3}}
\put(5.75,4.85){\Small{$\mathbf{p}$}}
\put(3.75,2.85){\Small{$\mathbf{\Delta}$}}
\multiput(15,0.5)(1,0){11}{\circle*{.2}}
\multiput(15,1.5)(1,0){11}{\circle*{.2}}
\multiput(15,2.5)(1,0){11}{\circle*{.2}}
\multiput(15,3.5)(1,0){11}{\circle*{.2}}
\multiput(15,4.5)(1,0){11}{\circle*{.2}}
\multiput(15,5.5)(1,0){11}{\circle*{.2}}
\multiput(15,6.5)(1,0){11}{\circle*{.2}}
\multiput(15,7.5)(1,0){11}{\circle*{.2}}
\multiput(15,8.5)(1,0){11}{\circle*{.2}}
\put(16,5.5){\circle*{.3}}
\put(15,5.85){\tiny{$(0,0)$}}
\thicklines
\put(16,5.5){\line(1,0){1}}
\put(17,5.5){\circle*{.3}}
\put(17,5.5){\line(1,-2){1}}
\put(21,3.5){\line(2,1){2}}
\put(21,3.5){\circle*{.3}}
\put(18,3.5){\circle*{.3}}
\put(18,3.5){\line(1,-1){1}}
\put(19,2.5){\circle*{.3}}
\put(19,2.5){\line(1,0){1}}
\put(23,4.5){\circle*{.3}}
\put(23,4.5){\line(2,-1){2}}
\put(20,2.5){\circle*{.3}}
\put(20,2.5){\line(1,1){1}}
\put(25,3.5){\circle*{.3}}
\put(25,3.5){\line(0,1){2}}
\put(25,5.5){\circle*{.3}}
\put(17.25,2.45){\Small{$\mathbf{p'}$}}
\end{picture}}
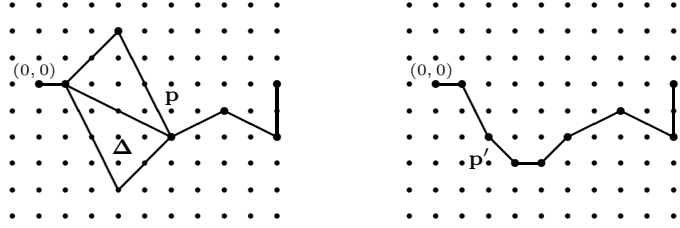
\centerline{Figure 2. The triangle $\Delta$ and a local convexification $\p'$ of a path $\p$.}

\vspace{.1in}
 
\noindent
\textit{Numbers.} If $\p=(\x_1, \ldots, \x_n)$ is any path we set $l(\p)=n$ (the \textit{length} of $\p$) and $|\p|=\sum \x_i$ (the \textit{weight} of $\p$). We will also need the notion of the  \textit{area} $a(\p)$ of a path $\p$ in $(\Z^2)^+$, which is defined as follows. Let $\p^{\#}$ be the unique convex path obtained by permuting the entries of $\p$. Then $a(\p)$ is the area of the polygon bounded by $\p$ and $\p^{\#}$~:
 
\centerline{
\begin{picture}(10,10)
\multiput(0,0.5)(1,0){11}{\circle*{.2}}
\multiput(0,1.5)(1,0){11}{\circle*{.2}}
\multiput(0,2.5)(1,0){11}{\circle*{.2}}
\multiput(0,3.5)(1,0){11}{\circle*{.2}}
\multiput(0,4.5)(1,0){11}{\circle*{.2}}
\multiput(0,5.5)(1,0){11}{\circle*{.2}}
\multiput(0,6.5)(1,0){11}{\circle*{.2}}
\multiput(0,7.5)(1,0){11}{\circle*{.2}}
\multiput(0,8.5)(1,0){11}{\circle*{.2}}
\put(1,5.5){\circle*{.3}}
\put(0,5.85){\tiny{$(0,0)$}}
\thicklines
\put(1,5.5){\line(1,0){1}}
\put(2,5.5){\circle*{.3}}
\put(2,5.5){\line(1,2){1}}
\put(3,7.5){\circle*{.3}}
\put(3,7.5){\line(1,-2){1}}
\put(4,5.5){\circle*{.3}}
\put(4,5.5){\line(1,-1){2}}
\put(6,3.5){\circle*{.3}}
\put(6,3.5){\line(2,1){2}}
\put(8,4.5){\circle*{.3}}
\put(8,4.5){\line(1,0){2}}
\put(10,4.5){\circle*{.3}}
\put(10,4.5){\line(0,1){1}}
\put(10,5.5){\circle*{.3}}
\put(1,5.5){\line(1,-2){1}}
\put(2,3.5){\circle*{.3}}
\put(2,3.5){\line(1,-1){2}}
\put(4,1.5){\circle*{.3}}
\put(4,1.5){\line(1,0){3}}
\put(5,1.5){\circle*{.3}}
\put(7,1.5){\circle*{.3}}
\put(7,1.5){\line(2,1){2}}
\put(9,2.5){\circle*{.3}}
\put(9,2.5){\line(1,2){1}}
\put(5.25,.7){\Small{$\mathbf{p}^{\#}$}}
\put(5.25,4.75){\Small{$\mathbf{p}$}}
\put(4.5,2.7){\Small{$a(\mathbf{p})$}}
\end{picture}}
\centerline{Figure 3. The area $a(\mathbf{p})$ of a path in $(\Z^2)^+$.}

 \vspace{.15in}
 
 The area function satisfies the following simple but important properties (see \cite[Lemma~5.6]{BS})~:
 
 \vspace{.05in}
 
 \begin{lem}\label{L:ap} Let $\p$ be a path in $(\Z^2)^+$. Then
 \begin{enumerate}
 \item[i)] $\p$ is convex if and only if $a(\p)=0$,
 \item[ii)] for any subpath $\p'$ of $\p$ we have $a(\p') \leq a(\p)$,
 \item[iii)] for any local convexification $\p'$ of $\p$ we have $a(\p') < a(\p)$.
 \end{enumerate}
\end{lem}
 
 \vspace{.15in}
 
 The notion of paths is relevant to us as it parametrizes bases of $\A$.  If $\p=(\x_1, \ldots, \x_n)$ is any path, we put $u_{\p}=u_{\x_1} \cdots u_{\x_n}$.

 \vspace{.1in}
 
 \begin{prop}[\cite{BS}, Prop.~6.2]\label{P:01} The set $\{u_{\p}\;|\; \p \in \textbf{Conv}\}$ is a $\mathbf{K}$-basis of $\A$.
 \end{prop}
 
 \vspace{.1in}
 
The above basis is nicely compatible with the various subalgebras of $\A$ we have considered. Namely, let $\textbf{Conv}^{+}$ and $\textbf{Conv}^-$ be defined respectively by the slope conditions
 $$-\pi/2 < \mu(\x_1) \leq \mu(\x_2) \leq \cdots \leq \mu(\x_n) \leq \pi/2$$
 and
 $$\pi/2 < \mu(\x_1) \leq \mu(\x_2) \leq \cdots \leq \mu(\x_n) \leq 3\pi/2.$$
 Then $\{u_{\p}\;|\; \p \in \textbf{Conv}^{\pm}\}$ forms a basis of $\A^{\pm}$. Similarly, let $\textbf{Conv}^{>}$ and $\textbf{Conv}^{<}$ be defined respectively by the slope conditions
 $$-\pi/2 < \mu(\x_1) \leq \mu(\x_2) \leq \cdots \leq \mu(\x_n) < \pi/2$$
 and
 $$\pi/2 < \mu(\x_1) \leq \mu(\x_2) \leq \cdots \leq \mu(\x_n) < 3\pi/2.$$
Then $\{u_{\p}\;|\; \p \in \textbf{Conv}^{>}\}$ and $\{u_{\p}\;|\; \p \in \textbf{Conv}^{<}\}$ form bases of $\A^{>}$ and $\A^{<}$ respectively.

\vspace{.1in}

The notion of convexification is related to the above bases via the following Lemma, proved in \cite[Section~6]{BS}~:

\vspace{.05in}

\begin{lem}\label{L:cv} For any nonconvex path $\p=(\x_1,\x_2)$ in $(\Z^2)^+$ we have
$$u_{\p}=u_{\x_1}u_{\x_2} \in \mathbf{K}^* u_{\x_1+\x_2} \oplus \bigoplus_{\mathbf{q}}\mathbf{K} u_{\mathbf{q}}$$
where $\mathbf{q}$ runs among all the set of all convexifications of $\p$ different from $(\x_1+\x_2)$.
\end{lem}

\vspace{.05in}

Iterating the above Lemma, we obtain

\vspace{.05in}

\begin{cor} For any path $\p$ in $(\Z^2)^+$ we have
$$u_{\p} \in \mathbf{K}^*u_{|\p|} \oplus \bigoplus_{\mathbf{q}} \mathbf{K} u_\mathbf{q}$$
where $\mathbf{q}$ runs among the set of all convex paths obtained from $\p$ by successive local convexifications different from $(|\p|)$.
\end{cor}

\vspace{.2in}

\paragraph{\textbf{5.2.}} We now start the proof of Theorem~\ref{T:3}. Let us temporarily denote by $\tA$ the algebra generated by the Fourier coefficients of  $\TT_{-1}(z), \TT_1(z), \TT_0^+(z)$ and $\TT_0^-(z)$, modulo relations (\ref{E:1}--\ref{E:5}) and (\ref{E:NEW}). To avoid any confusion with $\A$, we will denote by $\underline{u}_{\epsilon, l}$ or $\underline{\theta}_{(0,l)}$ the generators of $\tA$.
 
\vspace{.1in}
 
The fact that $\A$ is generated by the elements $\{u_{\pm 1, l}, l \in \Z\}$ and $\{\theta_{(0,l)}, l \in \Z^*\}$ is a consequence of \cite[Cor. 6.1]{BS}. That the functional equations (\ref{E:1}--\ref{E:5}) hold in $\A$ is also well known --see e.g. \cite[Thm. 3.3]{Kap} or \cite[Section~6.1]{BS}. The relation (\ref{E:NEW}) is easy to check directly. Hence there is a natural surjective algebra homomorphism $\phi: \tA \tto \A$ given by $\phi(\underline{u}_{\epsilon, l})=u_{\epsilon,l}$ and $\phi(\underline{\theta}_{(0,l)})=\theta_{(0,l)}$. The content of Theorem~\ref{T:3} is that the map $\phi$ is actually an isomorphism.

\vspace{.2in}

\paragraph{\textbf{5.3.}} We begin with the following Lemma. Let $\widetilde{\A}{}^>$, resp. $\tA{}^{<}$, resp. $\tA{}^{0}$, be the subalgebras of $\tA$ generated by $\{\underline{u}_{1,l}\;|\; l \in \Z\}$, resp. $\{\underline{u}_{-1,l}\;|\; l \in \Z\}$, resp. $\{\underline{u}_{0,l}\;|\; l \in \Z^*\}$.

\begin{lem}\label{L:1} The multiplication map 
\begin{equation}\label{E:m1}
m:  \tA{}^{>} \otimes  \tA{}^{0} \otimes  \tA{}^{<} \to  \tA
\end{equation}
is surjective.
\end{lem}

\noindent
\begin{proof} By the functional equation (\ref{E:5}) we have
\begin{equation}\label{E:p1}
[\ut_{1,l},\ut_{-1,n}] \in \tA{}^0.
\end{equation}
Similarly, (\ref{E:2}) is equivalent to
$$[\ut_{0,l},\ut_{1,n}]=\epsilon_l \ut_{1,n+l}, \qquad [\ut_{0,l},\ut_{-1,n}]=-\epsilon_l \ut_{-1,n+l}$$
(where $\epsilon_l=sgn(l) \in \{1, -1\}$) and hence
\begin{equation}\label{E:p2}
[\ut_{0,l},\ut_{-1,n}] \in \tA{}^{>}, \qquad [\ut_{0,l},\ut_{-1,n}] \in \tA{}^{<}.
\end{equation}
Using (\ref{E:p1}) and (\ref{E:p2}) it is easy to see that one may rewrite any monomial in the variables $\ut_{1,l}, \ut_{0,n}, \ut_{-1,m}$ as a (finite) linear combination
of similar monomials in which the variables $\ut_{1,l}, \ut_{0,n}, \ut_{-1,m}$ appear in that order. This proves the Lemma.\end{proof}

\noindent
\textbf{Remark.} Of course, it follows from Theorem~\ref{T:3} that (\ref{E:m1}) is actually an isomorphism.

\vspace{.1in}

Using Lemma~\ref{L:1}  and the triangular decomposition~(\ref{E:D2}) we are reduced to showing that the restriction of $\phi$ to each of the subalgebras
$\tA{}^{>}, \tA{}^{<}, \tA{}^{0}$ is injective. This is obvious for $\tA{}^{0}$ (which is a commutative polynomial algebra). The last two cases are of course identical so it suffices to deal with $\tA{}^{>}$. 

\vspace{.2in}

\paragraph{\textbf{5.4.}} Let us set
$$\A^{>}[r]=\bigoplus_{d \in \Z} \A^{>}[r,d], \qquad \tA{}^{>}[r]=\bigoplus_{d \in \Z} \tA{}^{>}[r,d].$$
The map $\phi$ restricts to a collection of maps $\phi_r:\tA{}^{>}[r] \tto \A^{>}[r]$. We will show by induction on $r$ that $\phi_r$ is injective. This is clear for $r =1$. For $r=2$, we have the functional equation (\ref{E:3}), which gives
\begin{equation}\label{E:6}
\begin{split}
-\ut_{1,n}\ut_{1,m}+&\alpha \ut_{1,n-1}\ut_{1,m+1} -\beta \ut_{1,n-2}\ut_{m+2} +\ut_{1,n-3}\ut_{1,m+3}\\
&=-\ut_{1,m}\ut_{1,n}+\beta \ut_{1,m+1}\ut_{1,n-1} -\alpha \ut_{1,m+2}\ut_{n-2} +\ut_{1,m+3}\ut_{1,n-3}
\end{split}
\end{equation}
where we have set $\alpha=(\sigma\overline{\sigma})^{-1}(\sigma + \overline{\sigma}) + \sigma \overline{\sigma}$ and $\beta=\sigma +\overline{\sigma}+\sigma\overline{\sigma}$. Recall (see Proposition~\ref{P:01}) that a basis of $\A^{>}[2,d]$ is provided by the elements
\begin{equation}\label{E:7}
u_{2,d}, \qquad u_{1,l}u_{1,k} \qquad (\text{for\;}l+k=d, \; l \leq k)
\end{equation}
associated to the set of convex paths in $\textbf{Conv}^{>}$ of rank $2$. We may use (\ref{E:6}) to try to express a product
$\ut_{1,n}\ut_{1,m}$  with $n >m$ (i.e. corresponding to a \textit{nonconvex} path) as a linear combination of products $\ut_{1,k}\ut_{1,l}$ associated to ``more convex'' paths.
More specifically, (\ref{E:6}) allows us to express $\ut_{1,n}\ut_{1,m}$ as a linear combination of  elements $\ut_{1,k}\ut_{1,l}$ where $(k,l)$ belongs to either of the sets
$$\{(n-1,m+1), (n-2,m+2), (n-3,m+3)\},$$
$$\{(m+3,n-3),(m+2,n-2),(m+1,n-1),(m,n)\}.$$
Using this, one shows that $\widetilde{\U}^{\geq 1}_{\v,\t}[2,d]$ is linearly generated by elements
\begin{equation}\label{E:8}
\begin{split}
\ut_{1,k}\ut_{1,l}, \qquad& l+k=d, \;k \leq l +1 \qquad (d \text{\;odd}),\\
\ut_{1,k}\ut_{1,l}, \qquad &l+k=d, \;k \leq l +2 \qquad (d \text{\;even}).
\end{split}
\end{equation}
Comparing (\ref{E:7}) and (\ref{E:8}) and using the fact that $\phi_2$ is surjective we deduce that $\phi_2$ is an isomorphism (and that all the elements in (\ref{E:8}) are actually linearly independent).

\vspace{.2in}

\paragraph{\textbf{5.5.}} Let us now fix some $r>2$, $d \in \Z$ and let us assume that $\phi_{r'}$ is an isomorphism for all $r'<r$. For any $\x =(s,l) \in (\Z^2)^+$ with $0 < s <r$ we put $\ut_{\x}=\phi_{s}^{-1}(t_{\x})$. Our first task will be to define a canonical elements in $\tA{}^>$ lifting $u_{(r,d)}, \theta_{(r,d)} \in \A^>$ for $d \in \Z$. For this we introduce the following notion~: 

\vspace{.1in}

\noindent
\textit{Minimal paths.} We say that an affine line in $\mathbb{R}^2$ is rational if it intersects $\Z^2$ in more than one (and hence infinitely many) points. Let
$\z \in (\mathbb{Z}^2)^>$. Let $L$ denote the rational line going through $(0,0)$ and $\z$, and let $L'$ be the rational line parallel to $L$ lying above $L$ and closest
to $L$. A \textit{minimal path of weight $\z$} is by definition a path in $(\Z^2)^>$ of the form $\mathbf{p}=(\x,\z-\x)$ where $\x$ is a point in $L'$.

\centerline{
\begin{picture}(20,20)
\multiput(0,0.5)(4,0){6}{\circle*{.2}}
\multiput(0,4.5)(4,0){6}{\circle*{.2}}
\multiput(0,8.5)(4,0){6}{\circle*{.2}}
\multiput(0,12.5)(4,0){6}{\circle*{.2}}
\multiput(0,16.5)(4,0){6}{\circle*{.2}}
\put(0,12.5){\circle*{.3}}
\put(-0.75,11.5){\tiny{$(0,0)$}}
\thicklines
\put(0,12.5){\line(3,-1){12}}
\put(5,8.5){$L$}
\put(5,12.5){$L'$}
\put(0,12.5){\line(2,-1){16}}
\multiput(16,4.5)(2,-1){3}{\line(2,-1){1}}
\multiput(0,14.5)(2,-1){11}{\line(2,-1){1}}
\put(16,4.5){\circle*{.3}}
\put(12,8.5){\circle*{.3}}
\put(11.75,9){\Small{$\mathbf{x}$}}
\put(12,8.5){\line(1,-1){4}}
\put(15.75,3.5){\Small{$\mathbf{z}$}}
\end{picture}}

\vspace{.1in}

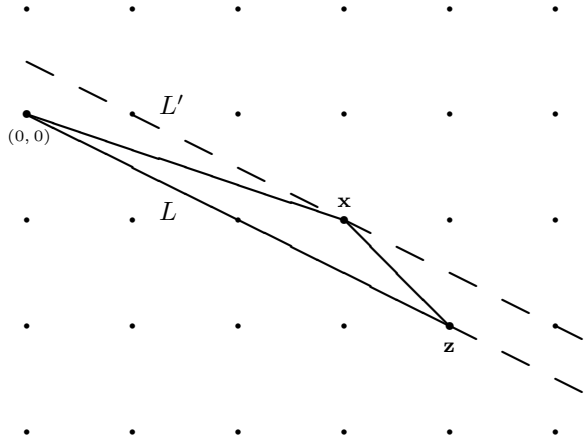
\centerline{Figure 4. The choice of a minimal path $\p_0$ of weight $\mathbf{z}$ .}

\vspace{.1in}

\begin{prop}\label{P:minexist} Let $\z=(r,d) \in (\Z^2)^>$ and assume that $r>2$. Then there exists a minimal path of weight $\z$.
\end{prop}
\begin{proof} Set $\mathbf{w}=\frac{1}{deg(\z)}\z$, a primitive vector. Let $\y$ be any element of $L' \cap \Z^2$. Then $L' \cap \Z^2=\{\y + n\w\;|\; n \in \Z\}$.
By definition, minimal paths of weight $\z$ bijectively correspond to elements of $L' \cap \Z^2$ lying in the interior of the strip $S$ bounded by the vertical lines through
$(0,0)$ and $\z$. The only case in which $S \cap L' \cap \Z^2$ is empty is when $deg(\z)=1$ and when the points $\{\y+n\z\}$ lie directly above the points $\{n\z\}$, i.e if $\y=(0,1)$. A simple application of Pick's formula to the parallelogram with vertices $(0,0), (0,1), \z, \z+(0,1)$ shows that this implies $r=1$, in contradiction with our assumptions.
\end{proof}

\vspace{.1in}

The following alternative characterization of minimal paths is useful.

\vspace{.1in}

\begin{prop}\label{P:minchar} A path $\mathbf{p}=(\x,\z-\x)$ in $(\Z^2)^>$ is minimal if and only if $deg(\x)=deg(\z-\x)=1$ and the triangle
${\Delta}_{\x,\z-\x}$ has no interior lattice point.\end{prop}
\begin{proof} Let $\p=(\x,\z-\x)$ be a path as above. We may draw equidistant lines $L'', L''',\ldots$ parallel to $L$ and $L'$, forming  a partition of the
upper half space bounded by $L$ and containing all lattice points in that upper half space.

\centerline{
\begin{picture}(25,20)
\put(5,8.5){$L$}
\put(5,12.5){$L'$}
\put(5,16.5){$L'''$}
\put(5,14.5){$L''$}
\multiput(0,0.5)(4,0){7}{\circle*{.2}}
\multiput(0,4.5)(4,0){7}{\circle*{.2}}
\multiput(0,8.5)(4,0){7}{\circle*{.2}}
\multiput(0,12.5)(4,0){7}{\circle*{.2}}
\multiput(0,16.5)(4,0){7}{\circle*{.2}}
\put(0,12.5){\circle*{.3}}
\put(-0.75,11.5){\tiny{$(0,0)$}}
\thicklines
\put(0,12.5){\line(2,-1){24}}
\multiput(0,14.5)(2,-1){12}{\line(2,-1){1}}
\multiput(0,16.5)(2,-1){12}{\line(2,-1){1}}
\multiput(8,16.5)(2,-1){8}{\line(2,-1){1}}
\multiput(4,16.5)(2,-1){10}{\line(2,-1){1}}
\put(24,0.5){\circle*{.3}}
\put(23.75,-.5){\Small{$\mathbf{z}$}}
\end{picture}}

\vspace{.1in}

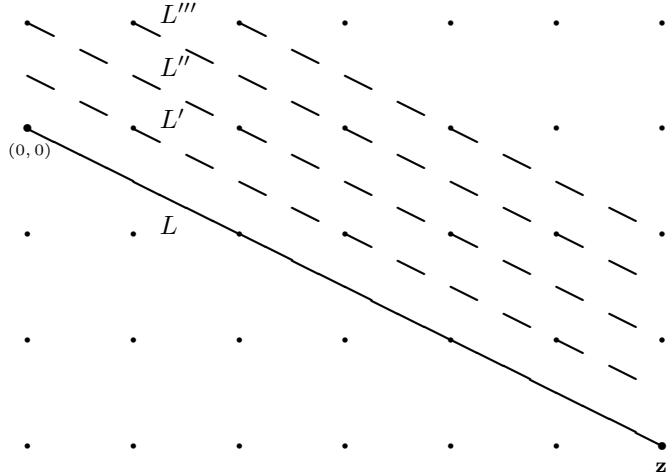
\centerline{Figure 5. The partition of the set of lattice points above $L$.}

\vspace{.15in}

We want to show that $\x$ belongs to $L'$. Suppose that this is not the case and let $\y \in L'$ be such that $\p'=(\y,\z-\y)$ is a minimal path. Consider the triangle
${\Delta}_{\x,\z-\x}$ By our hypothesis on $\p$, $\y$ lies either strictly to the left or strictly to the right of ${\Delta}_{\x,\z-\x}$. In the first case we have $\x-\y \in {\Delta}_{\x,\z-\x}$ while in the second case we have $\x+\z-\y \in {\Delta}_{\x,\z-\x}$, both of which bring contradiction.

\centerline{
\begin{picture}(25,20)
\put(4,13){\Small{$\mathbf{y}$}}
\multiput(0,0.5)(4,0){7}{\circle*{.2}}
\multiput(0,4.5)(4,0){7}{\circle*{.2}}
\multiput(0,8.5)(4,0){7}{\circle*{.2}}
\multiput(0,12.5)(4,0){7}{\circle*{.2}}
\multiput(0,16.5)(4,0){7}{\circle*{.2}}
\put(0,12.5){\circle*{.3}}
\put(-0.75,11.5){\tiny{$(0,0)$}}
\thicklines
\put(0,12.5){\line(1,0){4}}
\put(0,12.5){\line(4,-1){16}}
\put(0,12.5){\line(2,-1){24}}
\put(16,8.5){\line(1,-1){8}}
\multiput(0,14.5)(2,-1){12}{\line(2,-1){1}}
\multiput(0,16.5)(2,-1){12}{\line(2,-1){1}}
\multiput(8,16.5)(2,-1){8}{\line(2,-1){1}}
\multiput(4,16.5)(2,-1){10}{\line(2,-1){1}}
\put(24,0.5){\circle*{.3}}
\put(4,12.5){\circle*{.3}}
\put(15.75,9){\Small{$\mathbf{x}$}}
\put(16,8.5){\circle*{.3}}
\put(4,12.5){\line(5,-3){20}}
\put(23.75,-.5){\Small{$\mathbf{z}$}}
\end{picture}}

\vspace{.1in}

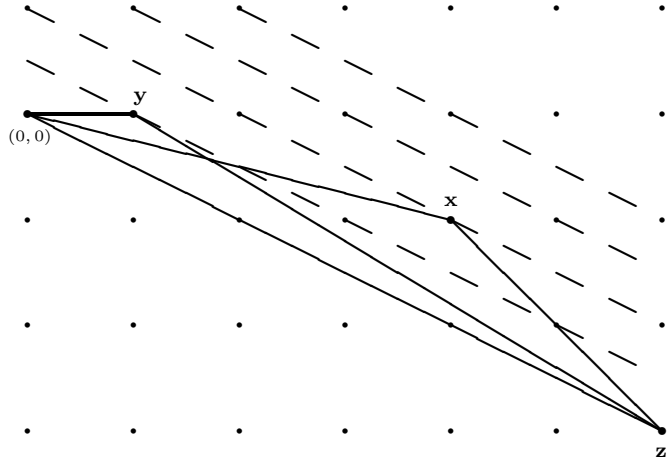
\centerline{Figure 6. The case when $\mathbf{y}$ lies to the left of ${\Delta}_{\x,\z-\x}$.}

\vspace{.1in}

\end{proof}

\vspace{.1in}

Note that if $\p=(\x,\z-\x)$ is a minimal path in $(\Z^2)^>$ of weight $\z$ then $[u_{\x},u_{\z-\x}]=\frac{1}{\a_1}u_{\z}$, by definition of $\A$. The idea here is to
use this in order to \textit{define} a suitable element $\ut_{\z} \in \tA{}^>[r]$ for $\z=(r,d)$; this element will be the lift of $u_{\z}$ to $\tA$. For this approach to make sense, we need the following result. Recall that we have fixed some $r >0$ and that we are arguing by induction, assuming that $\phi_{r'}$ is an isomorphism for all $r'<r$ (see section~3.4.).

\vspace{.1in}

\begin{prop}\label{P:defmin} Fix $\z=(r,d)$. For any two minimal paths $\p=(\x,\z-\x)$ and $\p'=(\y,\z-\y)$ it holds 
\begin{equation}\label{E:minprop}
[\ut_{\x}, \ut_{\z-\x}]=[\ut_{\y},\ut_{\z-\y}].
\end{equation}
\end{prop}
\begin{proof} Set $l=deg(\z)$ and $\w=\frac{1}{l}\z$. We first assume that $\w=(1,u)$ for some $u \in \Z$. Acting by the unipotent group ${\tiny{{\begin{pmatrix} 1 & \Z \\ 0 & 1 \end{pmatrix}}}}$, we may further assume that $u=0$, i.e. that $\z=(l,0)$. The minimal paths of weight $\z$ are then the paths $\big((i,1),(l-i,1)\big)$ for $i=1, \ldots, l-1$. If $l=3$ then applying $ad(\ut_{1,0})$ to the equality 
$[\ut_{1,1},\ut_{1,-1}]=\frac{1}{\a_1} \underline{\theta}_{2,0}$
and using the relations
\begin{equation}\label{E:two}
[\ut_{1,0},\ut_{1,1}]=-\ut_{2,1}, \qquad [\ut_{1,0},\ut_{1,-1}]=\ut_{2,-1}, \qquad [\ut_{1,0}, \ut_{2,0}]=0
\end{equation}
we obtain $-[\ut_{2,1},\ut_{1,-1}]+[\ut_{1,1},\ut_{2,-1}]=0$, which is (\ref{E:minprop}) in this case. Note that the last equality in (\ref{E:two}) holds by
virtue of (\ref{E:NEW}). Suppose now that $l>3$. By the induction hypothesis, $[\ut_{i,1},\ut_{l-1-i,-1}]=[\ut_{i-1,1},\ut_{l-i,-1}]$ for all $i=1, \ldots, l-2$. Applying 
$ad(\ut_{1,0})$ and using the relations $[\ut_{1,0},\ut_{i,1}]=-\ut_{i+1,1}, [\ut_{1,0},\ut_{i+1,-1}]=\ut_{i+1,-1}$ for all $i=1, \ldots, l-2$ we get
\begin{equation}\label{E:three}
[\ut_{i,1},\ut_{l-i,-1}]-[\ut_{i+1,1},\ut_{l-i-1,-1}]=[\ut_{i+1,1},\ut_{l-i-1,-1}]-[\ut_{i+2,1},\ut_{l-i-2,-1}]
\end{equation}
for $i=1, \ldots, l-3$. We need one more relation in order to be able to deduce that $[\ut_{i,1}\ut_{l-i,-1}]=[\ut_{j,1},\ut_{l-j,-1}]$ for all $i,j =1, \ldots, l-1$. If $l=4$ 
then we get this relation by applying $ad(\ut_{2,0})$ to the equality $[\ut_{1,1},\ut_{1,-1}]=\frac{1}{\a_1}\underline{\theta}_{2,0}$; if $l>4$ then we likewise get
this missing relation by applying $ad(\ut_{2,0})$ to
$[\ut_{i,1}\ut_{l-2-i,-1}]=[\ut_{j,1},\ut_{l-2-j,-1}]$ for all $i,j =1, \ldots, l-3$. This proves Proposition~\ref{P:defmin} when $\w$ is of rank one. The cases with $rk(\w) \geq 2$ may be dealt with in a similar fashion. If $l=1$ then there is a unique minimal path of weight $\z$ and there is nothing to prove. If $l=2$ then there are two minimal paths
$\p=(\x,\z-\x), \p'=(\x+\w, \w-\x)$; we have $[\ut_{\x},\ut_{\w}]=\ut_{\x+\w}, [\ut_{\w-\x},\ut_{\w}]=-\ut_{\z-\x}$ and
\begin{equation}\label{E:deffin}
[\ut_{\x}, \ut_{\w-\x}]=\ut_{\w}.
\end{equation}
The relation $[\ut_{\x}, \ut_{\z-\x}]=[\ut_{\x+\w},\ut_{\w-\x}]$ follows from applying $ad(\ut_{\w})$ to (\ref{E:deffin}). The cases with $l \geq 3$ are deduced exactly as
fin the situation $\w=(1,0)$ above~: we apply $ad(\ut_{\w})$ and $ad(\ut_{2\w})$ to the relevant equalities for $l-1$ and $l-2$. We leave the details to the reader.
\end{proof}

\vspace{.1in}

The above Proposition allows us to unambiguously define an element $\underline{\theta}_{(r,d)} \in \tA{}^>[r]$ for every $d \in \Z$ by setting
\begin{equation}
\underline{\theta}_{(r,d)}=\alpha_1 [\ut_{\x}, \ut_{(r,d)-\x}]
\end{equation}
for any minimal path $\p=(\x,(r,d)-\x)$. 

\vspace{.2in}

\paragraph{\textbf{5.6.}} If $\mathbf{s}=(\x_1, \ldots, \x_l)$ is a sequence in $(\Z^2)^>$ satisfying $rk(\x_i) \leq r$ for all $i$ then we put $\ut_{\mathbf{s}}=\ut_{\x_1} \cdots \ut_{\x_l}$. Clearly, $\tA{}^>[r]$ is linearly spanned by such elements $\ut_{\mathbf{s}}$--it suffices for instance to take all $\x_i$ of the form $(1,l_i)$. Recall from Section~1.2. the equivalence relation defined on the set of sequences.

\vspace{.1in}

\begin{prop}\label{P:seq} If $\mathbf{s}, \mathbf{s}'$ are two equivalent sequences in $(\Z^2)^>$ of weight $(r,d)$ then $\ut_{\mathbf{s}}=\ut_{\mathbf{s}'}$.
\end{prop}
\begin{proof} If $l(\mathbf{s}) >2$ then the weight of any subsequence $(\x_i, \x_{i+1})$ is of rank $r'<r$. Hence if
$\alpha(\x_i)=\alpha(\x_{i+1})$ then $\ut_{\x_i}\ut_{\x_{i+1}}=\ut_{\x_{i+1}}\ut_{\x_i}$ because $\phi_{r'}$ is an isomorphism. Therefore $\ut_{\mathbf{s}}=\ut_{\mathbf{s}'}$. Therefore the only case which we need to consider is
that of $\mathbf{s}=(\x_1, \x_2)$ with $\x_1 \neq \x_2$ while $\alpha(\x_1)=\alpha(\x_2)$. If $r=3$ then necessarily
$\{\x_1, \x_2\}=\{(1,l), (2,2l)\}$ for some $l \in \Z$, and the fact that $[\ut_{(1,l)}, \ut_{(2,2l)}]=0$ is a consequence of relation (\ref{E:NEW}). For a general $r$, choose a minimal path $(\z,\x_1-\z)$ of weight $\x_1$. We have $\underline{\theta}_{\x_1}=\a_1[\ut_\z, \ut_{\x_1-\z}]$. The two paths $\p'=(\z, \x_1+\x_2-\z)$ and
$\p''=(\z+\x_2,\x_1-\z)$ being both minimal of weight $\x_1+\x_2$ we have $[\ut_{\z},\ut_{\x_1+\x_2-\z}]=[\ut_{\z+\x_2},\ut_{\x_1-\z}]$ by Proposition~\ref{P:defmin}.

\centerline{
\begin{picture}(25,20)
\put(8,9){\Small{$\mathbf{z}$}}
\multiput(0,0.5)(4,0){7}{\circle*{.2}}
\multiput(0,4.5)(4,0){7}{\circle*{.2}}
\multiput(0,8.5)(4,0){7}{\circle*{.2}}
\multiput(0,12.5)(4,0){7}{\circle*{.2}}
\multiput(0,16.5)(4,0){7}{\circle*{.2}}
\put(0,4.5){\circle*{.3}}
\put(-0.75,3.5){\tiny{$(0,0)$}}
\thicklines
\put(0,4.5){\line(3,1){24}}
\put(0,4.5){\line(2,1){8}}
\put(0,4.5){\line(5,2){20}}
\put(8,8.5){\line(4,1){16}}
\put(24,12.5){\circle*{.3}}
\put(8,8.5){\circle*{.3}}
\put(19.75,13){\Small{$\mathbf{z}+\mathbf{x}_2$}}
\put(20,12.5){\circle*{.3}}
\put(12,8.5){\circle*{.3}}
\put(11.75,7.5){\Small{$\mathbf{x}_1$}}
\put(20,12.5){\line(1,0){4}}
\put(23.75,11.5){\Small{$\mathbf{x}_1+\mathbf{x}_2$}}
\end{picture}}

\vspace{.1in}

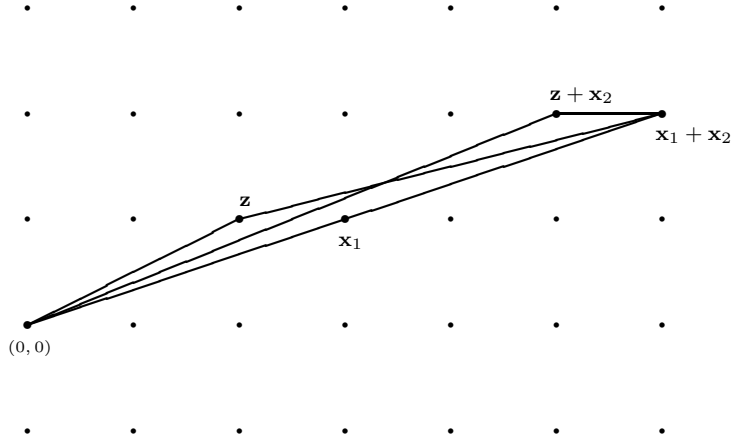
\centerline{Figure 7. The two minimal paths $\mathbf{p}$ and $\mathbf{p}'$.}

\vspace{.1in}

But then
\begin{equation}\label{E:four}
\begin{split}
[\underline{\theta}_{\x_1},\ut_{\x_2}]&=\a_1 \big( [[\ut_{\z},\ut_{\x_1-\z}],\ut_{\x_2}]\big)\\
&=\a_1 \big( [\ut_{\z},[\ut_{\x_1-\z},\ut_{\x_2}]] + [[\ut_{\z},\ut_{\x_2}],\ut_{\x_1-\z}]\big)\\
&=\a_1 \big( -[\ut_{\z},\ut_{\x_1+\x_2-\z}] + [\ut_{\z+\x_2},\ut_{\x_1-\z}]\big)=0.
\end{split}
\end{equation}
Since $\underline{\theta}_{\x_1}=\a_{deg(\x_1)} \ut_{\x_1} +h$ with $h$ being a linear combination of paths in $\mathbb{R}\x_1$ of length at least two, we
deduce from (\ref{E:four}) that $[\ut_{\x_1},\ut_{\x_2}]=0$ as wanted.
\end{proof}

Using the above Proposition we may now unambiguously define an element $\ut_{\p} \in \tA{}^>[r]$ for every \textit{path} $\mathbf{p}$ in $(\Z^2)^>$ of weight $(r,d)$. 
Indeed, if $\mathbf{p}=(\x_1, \ldots, \x_l)$ with $l \geq 2$ then we set $\ut_{\p}=\ut_{\x_1} \cdots \ut_{\x_l}$. If $l=1$, i.e. if $\p=(r,d)$ then one may use the relation (\ref{E:formulatheta}) and the construction of $\underline{\theta}_{(r,d)}$ provided by Prop.~\ref{P:defmin} in order to define an element $\ut_{r,d}$.

\vspace{.2in}

\paragraph{\textbf{5.7.}} After having constructed a canonical element $\ut_{\p} \in \tA{}^>[r]$ for every path $\p$ of weight $(r,d)$, we may proceed with the proof of the induction step.
We have
$$\A^{>}[r,d]= \bigoplus_{\p \in C} \mathbf{K} u_{\p},$$
where $C$ is the set of convex paths $\p \in  \textbf{Conv}^{>} $ of weight $(r,d)$.
We denote by $\widetilde{J}$ the subspace of $\tA{}^>[r,d]$ spanned by the $\ut_{\p}$ for $\p \in C$. By construction, $\phi_r$ restricts to an isomorphism $\widetilde{J} \stackrel{\sim}{\to} \A^>[r,d]$. Hence it suffices to show that $\widetilde{J}=\tA{}^>[r,d]$. 
In other words, we have to prove that every path $\p$ (not necessarily convex) we have
\begin{equation}\label{E:PPP1}
\ut_{\p} \in \widetilde{J}.
\end{equation}
Equation (\ref{E:PPP1}) will be established by an induction on the area $a(\p)$. If $a(\p)=0$ then by Lemma~\ref{L:ap} i) the path $\p$ is convex and hence $\ut_{\p} \in \widetilde{J}$ by definition. Now let $\p \in \textbf{Path}^{>}$ be as above and let us assume that (\ref{E:PPP1}) holds for any $\p'$ with $a(\p')<a(\p)$. At this point we break up the argument into several cases~:

\vspace{.15in}

\noindent
\textit{Case i)} Assume $l(\p) >2$. Then there exists a nonconvex subpath of length $2$, say $(\x_i, \x_{i+1})$. Since $l(\p)>2$ we have $\x_i + \x_{i+1} =(r',d')$ with $r'<r$. Because $\phi_{r'}$ is an isomorphism we have, using Lemma~\ref{L:cv}
$$\ut_{\x_i}\ut_{\x_{i+1}}=\phi_{r'}^{-1}(u_{\x_i}u_{\x_{i+1}})=\phi_{r'}^{-1}\big(\sum_{\mathbf{q}}\beta_{\mathbf{q}} u_{\mathbf{q}}\big)=\sum_{\mathbf{q}}\beta_{\mathbf{q}} \ut_{\mathbf{q}}$$
where $\mathbf{q}$ runs over a certain (finite) set of paths in $\textbf{Conv}^{>}$ of weight $(r',d')$. Therefore, $\ut_{\p}=\sum_{\mathbf{q}} \beta_{\mathbf{q}}\ut_{\mathbf{q}'}$ where by definition $\mathbf{q}'$ is the concatenation of $(\x_1, \ldots, \x_{i-1}), \mathbf{q}$ and $(\x_{i+2}, \ldots, \x_{l(\mathbf{p})})$. By Lemma~\ref{L:ap} iii) we have $a(\mathbf{q}') < a(\mathbf{p})$ for all paths $\mathbf{q}'$ appearing in this fashion, and we may use the induction hypothesis to deduce that $\ut_{\p}$ belongs to $\widetilde{J}$.

\vspace{.15in}

\noindent
\textit{Case ii)} Assume that $l(\p)=2$, i.e. that $\p=(\x, (r,d)-\x)$. Let $R$ be the region (parallelogram) in $(\Z^2)^+$ bounded by the segment  $[0,(r,d)]$, its parallel segment going through $\x$ and the two vertical lines going through $(0,0)$ and $(r,d)$. Let $\Delta$, $\Delta_{\x}$ and $\Delta'$ stand for the left, central and right triangles lying inside $R$ (see Figure~8 below). 

\centerline{
\begin{picture}(20,18)
\multiput(0,0.5)(4,0){7}{\circle*{.2}}
\multiput(0,4.5)(4,0){7}{\circle*{.2}}
\multiput(0,8.5)(4,0){7}{\circle*{.2}}
\multiput(0,12.5)(4,0){7}{\circle*{.2}}
\multiput(0,16.5)(4,0){7}{\circle*{.2}}
\put(16,12.5){\circle*{.3}}
\put(3.5,7.5){\tiny{$(0,0)$}}
\thicklines
\put(4,8.5){\line(3,1){12}}
\put(4,8.5){\circle*{.3}}
\put(4,8.5){\line(4,-1){16}}
\put(20,4.5){\circle*{.3}}
\put(15.5, 13){\Small{$\mathbf{x}$}}
\put(6, 12.75){\Small{${\Delta}$}}
\put(12.5, 9.75){\Small{${\Delta_{\x}}$}}
\put(17.75, 9.75){\Small{${\Delta}'$}}
\put(16,12.5){\line(1,-2){4}}
\put(16,12.5){\line(4,-1){4}}
\put(20,4.5){\line(0,1){7}}
\put(4,8.5){\line(0,1){7}}
\put(4,15.5){\line(4,-1){12}}
\put(19.25,3.75){\Small{$(r,d)$}}
\end{picture}}

\vspace{.1in}

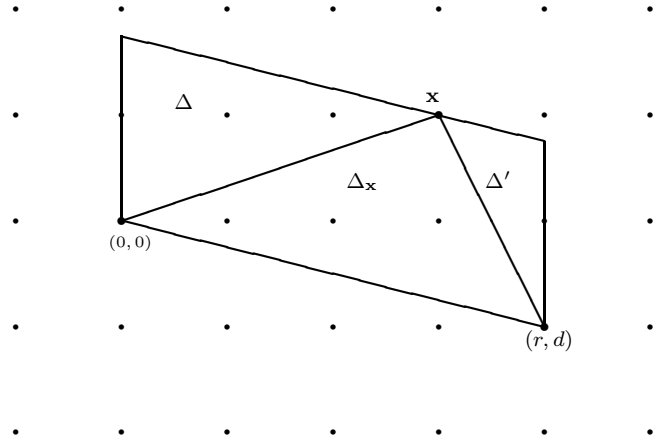
\centerline{Figure 8. The region $R$ and the three triangles $\Delta$, $\Delta'$ and $\Delta_\mathbf{x}$.}

\vspace{.1in}

We break again this case into several distinct ones~:

\vspace{.1in}

\noindent
\textit{Case ii) a).} Assume that there is a lattice point $\y$ in one of the two triangles $\Delta$ and $\Delta'$. We allow this point to lie on the left boundary of ${\Delta}$, resp. on the right boundary of ${\Delta}'$, but not on the boundary in common with ${\Delta}_\x$ nor on the top boundary. Without loss of generality, we may assume that $\y \in \Delta$, the other case $\y \in \Delta'$ being completely identical.
Consider the path $\mathbf{q}=(\y, (r,d)-\y)$. By construction we have 
$$a(\mathbf{q})=|det(\y, (r,d)-\y)|=|det(\y,(r,d))| < |det(\x,(r,d))|=a(\p)$$
and hence $\ut_{\mathbf{q}} \in \widetilde{J}$ by our induction hypothesis. Consider also the path $\mathbf{n}=(\y,\x-\y, (r,d)-\x)$. Note that $\mathbf{n}$ is entirely contained in the region $R$. We will show that
\begin{equation}\label{E:P41}
\ut_{\mathbf{p}} \equiv \alpha\ut_{\mathbf{n}} \equiv \alpha' \ut_{\mathbf{q}}\;(\text{mod}\;(\widetilde{J}))
\end{equation}
for some $\alpha, \alpha' \in \mathbf{K}$, thereby proving that $\ut_{\p} \in \widetilde{J}$.

\vspace{.05in}

First of all, we have by Lemma~\ref{L:cv} 
$$u_{\y}u_{\x-\y} = \beta u_{\x} + \sum_{\mathbf{m}} \beta_{\mathbf{m}} u_{\mathbf{m}}$$
where $\beta \neq 0$ and where $\mathbf{m}$ runs among the set of convexifications of $(\y,\x-\y)$ which are distinct from $(\x)$. Since $\phi_{r'}$ is an isomorphism for all $r'<r$, the same equation holds in $\widetilde{\U}^{>}_{\v,\t}$ (i.e. with $u$s replaced by $\ut$s). As a consequence, we have
\begin{equation}\label{E:P51}
\begin{split}
\ut_{\mathbf{n}}&=\ut_{\y}\ut_{\x-\y}\ut_{(r,d)-\x}\\
&=\beta \ut_{\x}\ut_{(r,d)-\x} + \sum_{\mathbf{m}} \beta_{\mathbf{m}} \ut_{\mathbf{m}} \ut_{(r,d)-\x}\\
&=\beta \ut_{\p} + \sum_{\mathbf{m}} \beta_{\mathbf{m}} \ut_{\mathbf{m}'}
\end{split}
\end{equation}
where $\mathbf{m}'$ is the concatenation of $\mathbf{m}$ and $((r,d)-\x)$. Observe that any segment $\mathbf{z}_i$ appearing in a convexification $\mathbf{m}$ of $(\y,\x-\y)$ satisfies $\alpha(\mathbf{z}_i) \geq \alpha(\x-\y) \geq \alpha ((r,d)-\x)$ (see Figure~9 below). It follows that for all such $\mathbf{m}$ we have $a(\mathbf{m}')=a((\x,(r,d)-\x))=a(\p)$. However, by construction we also have $l(\mathbf{m}')>2$ and hence by Case i) above $\ut_{\mathbf{m}'} \in \widetilde{J}$. This proves the first congruence in (\ref{E:P41}). 

\centerline{
\begin{picture}(20,20)
\multiput(0,0.5)(4,0){7}{\circle*{.2}}
\multiput(0,4.5)(4,0){7}{\circle*{.2}}
\multiput(0,8.5)(4,0){7}{\circle*{.2}}
\multiput(0,12.5)(4,0){7}{\circle*{.2}}
\multiput(0,16.5)(4,0){7}{\circle*{.2}}
\put(16,12.5){\circle*{.3}}
\put(8,12.5){\circle*{.3}}
\put(3.5,7.5){\tiny{$(0,0)$}}
\put(8,13){\Small{$\y$}}
\put(11.5,9.5){\Small{$\mathbf{m}$}}
\put(4,8.5){\line(1,1){4}}
\thicklines
\put(4,8.5){\line(3,1){12}}
\put(4,8.5){\line(1,0){4}}
\put(8,12.5){\line(1,0){8}}
\put(8,8.5){\line(2,1){8}}
\put(4,8.5){\circle*{.3}}
\put(4,8.5){\line(4,-1){16}}
\put(20,4.5){\circle*{.3}}
\put(15.5, 13){\Small{$\mathbf{x}$}}
\put(16,12.5){\line(1,-2){4}}
\put(16,12.5){\line(4,-1){4}}
\put(20,4.5){\line(0,1){7}}
\put(4,8.5){\line(0,1){7}}
\put(4,15.5){\line(4,-1){12}}
\put(19.25,3.75){\Small{$(r,d)$}}
\end{picture}}

\vspace{.1in}

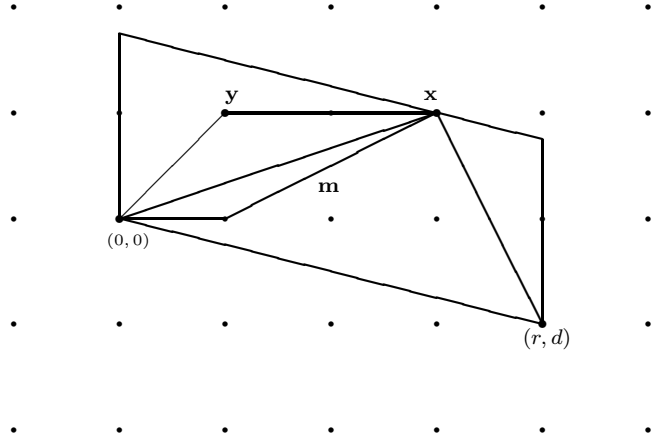
\centerline{Figure 9. A lattice point $\mathbf{y}$ in $\Delta$ and a convexification $\mathbf{m}$ of $(\y,\x-\y)$.}

\vspace{.15in}

The second congruence is proved in a similar fashion. Using Lemma~\ref{L:cv} again, we have
\begin{equation}\label{E:P512}
\begin{split}
\ut_{\mathbf{n}}&=\ut_{\y}\ut_{\x-\y}\ut_{(r,d)-\x}\\
&=\beta' \ut_{\y}\ut_{(r,d)-\y} + \sum_{\mathbf{m}} \beta'_{\mathbf{m}} \ut_{\y}\ut_{\mathbf{m}} \\
&=\beta' \ut_{\mathbf{q}} + \sum_{\mathbf{m}} \beta'_{\mathbf{m}} \ut_{\mathbf{m}'}
\end{split}
\end{equation}
where $\beta' \neq 0$, $\mathbf{m}$ runs among the set of convexifications of $(\x-\y, (r,d)-\x)$ distinct from $((r,d)-\y)$, and where $\mathbf{m}'$ is the concatenation of $\y$ and $\mathbf{m}$. This time we have $\alpha(\mathbf{z}_i) \leq \alpha(\x-\y) \leq \alpha(\y)$ for all segments $\mathbf{z}_i$ appearing in the convexifications $\mathbf{m}$, and hence $a(\mathbf{m}') =a((\y, (r-d)-\y))=a(\mathbf{q}) < a(\p)$. We deduce that $\ut_{\mathbf{m}'} \in \widetilde{J}$ and the second congruence in (\ref{E:P41}) is proved.

\vspace{.1in}

\noindent
\textit{Case ii) b).} Assume that there are no interior points in either of the triangles $\Delta$, $\Delta'$. Observe that the union of $\Delta'$ and the translate of $\Delta$ by the vector $(r,d)$ is equal to a reflection of the central triangle $\Delta_{\x}$ in $R$. Hence the triangle $\Delta_{\x}$ has no interior lattice points. It may however have lattice points on its boundary. Assume first that it has a lattice point on its bottom boundary, i.e. that $(r,d)$ is not a primitive vector. Then, by symmetry, there exists a lattice point, say $\mathbf{z}$ on the top boundary of one of $\Delta$ or $\Delta'$. Consider the path $\mathbf{q}=(\z,(r,d)-\z)$. Arguing as in case ii a) above, using the path of length three connecting $(0,0)$, $\z$, $\x$ and $(r,d)$, we get thet $\ut_{\p} \equiv \alpha \ut_{\mathbf{q}}\;(mod\; \tilde{J})$ for some $\alpha \in \mathbf{K}$. By our assumption, the vectors $\z$ and $(r,d)-\z$ are both primitive, and there are no interior points in the triangle $\Delta_{\z,(r,d)-\z}$. Hence by Prop~\ref{P:minchar} the path $\mathbf{q}$ is minimal and therefore $\ut_{\mathbf{q}} \in \tilde{J}$. Next, let us assume that $deg((r,d))=1$. It is not possible to have $deg(\x)=deg((r,d)-\x)=2$, and if $deg(\x) \geq 2$ and $deg((r,d)-\x) \geq 3$ or if $deg(\x) \geq 3$ and $deg((r,d)-\x) \geq 2$ then there exists a lattice point in the interior of $\Delta_{\x}$, contrary to our assumption. Hence we have $deg(\x)=1$ or $deg((r,d)-\x)=1$. The two cases are symmetric, and so we will assume now that $deg(\x)=1$. The following picture gives an example of that situation~:

\centerline{
\begin{picture}(20,20)
\multiput(0,0.5)(4,0){7}{\circle*{.2}}
\multiput(0,4.5)(4,0){7}{\circle*{.2}}
\multiput(0,8.5)(4,0){7}{\circle*{.2}}
\multiput(0,12.5)(4,0){7}{\circle*{.2}}
\multiput(0,16.5)(4,0){7}{\circle*{.2}}
\put(0,12.5){\circle*{.3}}
\put(-0.5,11.5){\tiny{$(0,0)$}}
\thicklines
\put(0,12.5){\line(5,-2){20}}
\put(4,12.5){\line(5,-2){16}}
\put(0,12.5){\line(1,0){4}}
\put(4,12.5){\line(2,-1){16}}
\put(0,12.5){\line(0,1){1.6}}
\put(20,4.5){\line(0,1){1.6}}
\put(0,14.1){\line(5,-2){4}}
\put(4,12.5){\circle*{.3}}
\put(20,4.5){\circle*{.3}}
\put(3.5, 13){\Small{$\mathbf{x}$}}
\put(19.25,3.75){\Small{$(r,d)$}}
\end{picture}}

\vspace{.1in}

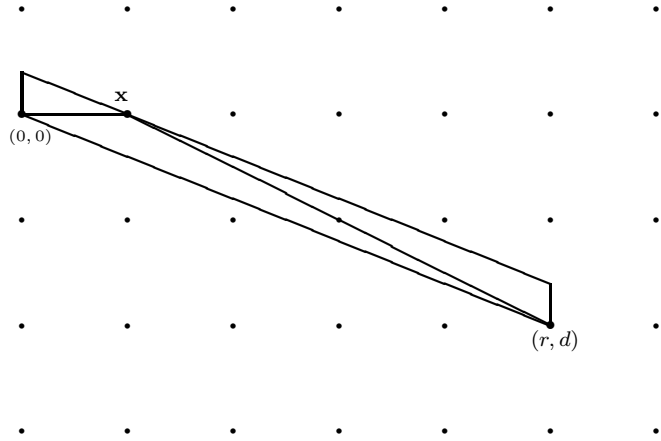
\centerline{Figure 10. The last case.}

\vspace{.1in}

Set $\w=\frac{1}{deg(\z-\x)}(\z-\x)$ and put $\y=\x-\w$. Let us write $\y=(s,l)$. Note that we may $s \leq 0$, but in any case $|s| <r$. Since
$\tA{}^<$ and $\tA{}^>$ are clearly isomorphic, our basic induction hypothesis that $\phi_{r'}: \tA{}^>[r'] \to \A{}^>[r']$ is an isomorphism for all $0 \leq r'<r$ holds for
$\tA{}^<$ as well, with $-r$ in place of $r$. Hence there exists $\ut_{\y} \in \tA$ satisfying $\phi(\ut_{\y})=u_{\y}$. Observe that, as there are no interior lattice
points in the triangles $\Delta, \Delta', \Delta_{\x}$ we have
\begin{equation}\label{E:last1}
[u_{\y},u_{\w}]=u_{\x}, \qquad [u_{\y},u_{\z-\x}]=u_{\z-\w}.
\end{equation}
We claim that
\begin{equation}\label{E:last2}
[\tu_{\y},\tu_{\w}]=\tu_{\x}, \qquad [\tu_{\y},\tu_{\z-\x}]=\tu_{\z-\w}.
\end{equation}
This follows directly from (\ref{E:last1}) and our induction hypothesis if $s \geq 0$. If $s <0$ note that by virtue of (\ref{E:5}) we have
$$[\tu_{\y}, \tu_{\x}], [\tu_{\y}, \tu_{\z-\x}] \in \bigoplus_{0 \leq, r',r'' <r} \tA{}^>[r'] \otimes \tA{}^<[-r''] =: V$$
while by our induction hypothesis the restriction of $\phi$ to $V$ is injective. From (\ref{E:last2}) we get
\begin{equation*}
\begin{split}
[\tu_{\x},\tu_{\z-\x}]=[[\tu_{\y},\tu_{\w}],\tu_{\z-\x}]=&[[\tu_{\y},\tu_{\z-\x}],\tu_{\w}] + [\tu_{\y},[\tu_{\w},\tu_{\z-\x}]]\\
&=[\tu_{\z-\w},\tu_{\w}].
\end{split}
\end{equation*}
But $(\z-\w,\w)$ is a minimal path hence $[\tu_{\z-\w},\tu_{\w}]=\frac{1}{\a_1}\underline{\theta}_{(r,d)} \in \widetilde{J}$. This proves that
$[\tu_{\x},\tu_{\z-\x}] \in \widetilde{J}$ as wanted. This finishes case ii) b), and brings the induction argument to a close. Theorem~\ref{T:3} is proved.
\qed

\vspace{.2in}

\section{Hopf algebra structure}

\vspace{.15in}

The elliptic Hall algebra $\A$ is naturally equipped with a coproduct, making it a (topological) bialgebra (see \cite{BS}). In terms of the Drinfeld generators, this coproduct takes the following form (see \cite[Lemma 4.1]{BS})~:
\begin{equation}
\begin{split}
\Delta(\mathbb{T}_1(z))&=\mathbb{T}_1(z) \otimes 1 + \mathbb{T}_0^+(z) \otimes \mathbb{T}_1(z),\\
\Delta(\mathbb{T}_{-1}(z))&=\mathbb{T}_{-1}(z) \otimes 1 + \mathbb{T}_{-1}(z) \otimes \mathbb{T}_0^-(z),\\
\Delta(\mathbb{T}_0^{\pm}(z))&=\mathbb{T}^{\pm}_0(z) \otimes \mathbb{T}_{0}^{\pm}(z).
\end{split}
\end{equation}

Note in particular that $\A^{\pm}$ are sub bialgebras of $\A$. The theory of Hall algebras also provides 
$\A^{\pm}$ with a (topological) antipode. We won't write it here. Finaly, note that the $SL(2,\Z)$-symmetry is broken if one takes the coproduct
into consideration; in other words, $SL(2,\Z)$ does not act by Hopf algebra automorphisms (however, the unipotent subgroup $\small{\begin{pmatrix}1 & \Z\\ 0 & 1\end{pmatrix}}$ does).

\vspace{.2in}

\centerline{\textbf{Acknowledgements}}

\vspace{.1in}

I would like to thank A. Tsymbaliuk for useful discussions and for pointing out the paper \cite{F2JM2} to me.

\vspace{.2in}

\small{}

\end{document}